\documentclass{amsart}
\usepackage{amsmath}
\usepackage{amssymb}
\usepackage{amsthm}
\usepackage[top=3cm,left=3cm,right=3cm,bottom=3cm]{geometry} 
\usepackage[bitstream-charter]{mathdesign}
\usepackage[T1]{fontenc}
\usepackage[mathscr]{eucal}
\usepackage[shortlabels]{enumitem}
\usepackage[all]{xy}
\usepackage{xcolor}
\usepackage{hyperref}
\numberwithin{equation}{section}
\theoremstyle{plain}
\newtheorem{theorem}{Theorem}[section]

\newtheorem{lemma}[theorem]{Lemma}
\newtheorem{corollary}[theorem]{Corollary}
\theoremstyle{definition}
\newtheorem{definition}[theorem]{Definition}

\newtheorem{notation}[theorem]{Notation}
\newtheorem{remark}[theorem]{Remark}
\newtheorem{example}[theorem]{Example}

\theoremstyle{plain}
\newtheorem{thm}{Theorem}

\newcommand{\mcal}[1]{\mathcal{#1}}
\newcommand{\mbb}[1]{\mathbb{#1}}
\newcommand{\mbf}[1]{\mathbf{#1}}
\newcommand{\mrm}[1]{\mathrm{#1}}

\newcommand{\mscr}[1]{\mathscr{#1}}

\newcommand{\Fun}{\operatorname{Fun}}
\newcommand{\Map}{\operatorname{Map}}
\newcommand{\cofib}{\operatorname{cofib}}
\newcommand{\cosk}{\operatorname{cosk}}
\newcommand{\colim}{\operatorname*{colim}}

\newcommand{\Corr}{\mathrm{Corr}}

\newcommand{\Gr}{\mathrm{Gr}}
\newcommand{\MGL}{\mathrm{MGL}}
\newcommand{\Mon}{\mathrm{Mon}}

\newcommand{\Perf}{\mathrm{Perf}}
\newcommand{\PMGL}{\mathrm{PMGL}}
\newcommand{\Proj}{\mathrm{Proj}}
\newcommand{\PSh}{\mathrm{PSh}}
\newcommand{\QCoh}{\mathrm{QCoh}}
\newcommand{\Sch}{\mathrm{Sch}}
\newcommand{\SH}{\mathrm{SH}}
\newcommand{\Shv}{\mathrm{Shv}}
\newcommand{\Sp}{\mathrm{Sp}}
\newcommand{\Spec}{\mathrm{Spec}}
\newcommand{\Sym}{\mathrm{Sym}}

\newcommand{\co}{\mathrm{co}}
\newcommand{\ex}{\mathrm{ex}}
\newcommand{\fqsm}{\mathrm{fqsm}}
\newcommand{\id}{\mathrm{id}}
\newcommand{\lax}{\mathrm{lax}}
\newcommand{\mot}{\mathrm{mot}}
\newcommand{\op}{\mathrm{op}}
\newcommand{\pbf}{\mathrm{pbf}}
\newcommand{\tr}{\mathrm{tr}}
\newcommand{\univ}{\mathrm{univ}}
\newcommand{\Zar}{\mathrm{Zar}}
\newcommand{\Cat}{\mathscr{C}\mathrm{at}}
\newcommand{\FQSm}{\mathscr{FQS}\mathrm{m}}
\newcommand{\Pic}{\mathscr{P}\mathrm{ic}}
\newcommand{\Vect}{\mathscr{V}\mathrm{ect}}

\title{Cohomology of the moduli stack of algebraic vector bundles}

\author{Toni Annala}
\address{School of Mathematics, Institute for Advanced Study, 1 Einstein Drive, 08540 Princeton, NJ, USA.}
\email{\href{mailto:tannala@ias.edu}{tannala@ias.edu}}
\thanks{The first author was support by the Vilho, Yrj\"o and Kalle V\"ais\"al\"a Foundation of the Finnish Academy of Science and Letters.}

\author{Ryomei Iwasa}
\address{Laboratoire de Math\'ematiques d'Orsay, Universit\'e Paris-Saclay, 307 rue Michel Magat, F-91405 Orsay.}
\email{\href{mailto:ryomei.iwasa@cnrs.fr}{ryomei.iwasa@cnrs.fr}}
\thanks{The second author was supported by the European Union's Horizon 2020 research and innovation programme under the Marie Sk\l{}odowska-Curie grant agreement No.\ 896517}

\begin{document}

\date{\today}

\begin{abstract}
Let $\Vect_n$ be the moduli stack of vector bundles of rank $n$ on derived schemes.
We prove that, if $E$ is a Zariski sheaf of ring spectra which is equipped with finite quasi-smooth transfers and satisfies projective bundle formula, then $E^*(\Vect_{n,S})$ is freely generated by Chern classes $c_1,\dotsc,c_n$ over $E^*(S)$ for any qcqs derived scheme $S$.
Examples include all multiplicative localizing invariants.
\end{abstract}

\maketitle

\tableofcontents

\section{Introduction}

In algebraic topology, the cohomology of the classifying spaces of unitary groups is fundamental:
for a complex oriented cohomology theory $E$ and $n\ge 0$, there is a canonical ring isomorphism 
\[
	E^*(BU(n)) \simeq \pi_*E[[c_1,\dotsc,c_n]],
\]
where $c_1,\dotsc,c_n$ are the universal Chern classes.
The goal of this paper is to establish its algebraic counterpart.
Examples of algebraic cohomology theories our results apply to are localizing invariants in the sense of \cite{BGT} such as algebraic $K$-theory and topological Hochschild homology, as well as non $\mbb{A}^1$-localized algebraic cobordism, which we define.
In order to work in this generality, we develop a version of motivic homotopy theory and show that all localizing invariants are representable there.
This would be the key computational step toward further study of algebraic $K$-theory and algebraic cobordism beyond $\mbb{A}^1$-homotopy invariance.\footnote{In the sequel \cite{AI}, the results and techniques obtained in this paper are applied to prove a universality of $K$-theory. Generalizations of the results in this paper are also established in op.\ cit.}

Let us start by discussing what the algebraic counterpart of complex oriented cohomology theories should be.
To detect orientations, we take the viewpoint of transfers; see \cite[\S1]{Qui} for the relation between complex orientations, cobordism, and transfers in algebraic topology.
On the algebraic side, it is shown in \cite{EHKSYb} that the algebraic cobordism $\MGL$ is universal among $\mbb{A}^1$-local motivic spectra with finite quasi-smooth transfers.
Taking this into account, we consider \textit{Zariski sheaves with finite quasi-smooth transfers} on derived schemes, which we call \textit{sheaves with transfers} for short, cf.\ Definition \ref{def:pst} and \ref{def:sheaf}.
In practice, we restrict to sheaves on the $\infty$-category $\Sch_S$ of derived schemes of finite presentation over a qcqs derived scheme $S$.
We remark that derived schemes are essential to formulate sheaves with finite quasi-smooth transfers.

Instead of the $\mbb{A}^1$-homotopy invariance as in Morel-Voevodsky's theory \cite{MV}, our basic input is the projective bundle formula.
We say that a sheaf $E$ with transfers on $\Sch_S$ \textit{satisfies projective bundle formula} or \textit{is pbf-local} if the map 
\[
	\iota_*\oplus p^* \colon E(\mbb{P}^{n-1}_X) \oplus E(X) \to E(\mbb{P}^n_X)
\]
is an equivalence for every $X\in\Sch_S$ and $n\ge 0$, where $\iota_*$ is the pushforward along a fixed linear embedding $\mbb{P}^{n-1}\to\mbb{P}^n$, which comes from the transfers of $E$.
All localizing invariants satisfy projective bundle formula, while they do not satisfy $\mbb{A}^1$-homotopy invariance in general.

To state our main result, we introduce some notations.
Let $\Shv^\tr(\Sch_S)$ be the $\infty$-category of sheaves with transfers on $\Sch_S$ (Definition \ref{def:sheaf}), $\Shv^\tr_\pbf(\Sch_S)$ its full subcategory spanned by pbf-local sheaves with transfers, and $\SH^\tr_\pbf(\Sch_S)$ the $\infty$-category of spectrum objects in $\Shv^\tr_\pbf(\Sch_S)$.
Then $\SH^\tr_\pbf(\Sch_S)$ has a canonical symmetric monoidal structure and an algebra object there is a sheaf of ring spectra which is equipped with transfers and satisfies projective bundle formula.
For example, a multiplicative localizing invariant yields an $\mbb{E}_\infty$-algebra in $\SH^\tr_\pbf(\Sch_S)$ (Corollary \ref{cor:mult}).
Let $\Vect_n$ be the moduli stack of vector bundles of rank $n$.
If $E$ is a presheaf of spectra and $X$ is an algebraic stack, then we write $E^*(X):=\pi_{-*}\Map(\Sigma^\infty_+X,E)$.

\begin{thm}[Corollary \ref{cor:vect}]\label{thm0:coh}
Let $S$ be a qcqs derived scheme and $n\ge 0$.
Let $E$ be a homotopy commutative algebra in $\SH^\tr_\pbf(\Sch_S)$.
Then there is a canonical ring isomorphism
\[
	E^*(\Vect_{n,S}) \simeq E^*(S)[[c_1,\dotsc,c_n]].
\]
\end{thm}

The isomorphism is well-known for oriented $\mbb{A}^1$-local motivic ring spectra, cf.\ \cite[Proposition 6.2]{NSO}.
For algebraic $K$-theory, the isomorphism was known when the base is a regular scheme as a special case of the $\mbb{A}^1$-local motivic result.
Topological Hochschild homology is not $\mbb{A}^1$-homotopy invariant, but it would be possible to prove the isomorphism directly by using the comparison with the de Rham-Witt complex; see \cite{Tot} for a related computation.
Our theorem extends these to results on general localizing invariants over general bases and gives a unified proof.

Theorem \ref{thm0:coh} follows from a comparison of the ``motivic'' homotopy type of $\Vect_n$ and that of the infinite grassmannian $\Gr_n$.
The cohomology of the latter is rather simple to calculate and we get Theorem \ref{thm0:coh}.
The comparison is stated as follows.
Let $\gamma^*$ be the left adjoint of the forgetful functor $\Shv^\tr(\Sch_S)\to\Shv(\Sch_S)$ and $L_\mot$ the localization functor enforcing projective bundle formula on sheaves with transfers.
Then:

\begin{thm}[Theorem \ref{thm:vect}]\label{thm0:vect}
For every qcqs derived scheme $S$ and $n\ge 0$, the canonical map
\[
	L_\mot\gamma^*\Gr_{n,S} \to L_\mot\gamma^*\Vect_{n,S}
\]
is an equivalence.
\end{thm}

We believe that our methods would give new insights into the theory of motives beyond $\mbb{A}^1$-homotopy invariance.
Over the past few years, there have been several attempts to study non $\mbb{A}^1$-local motivic phenomena as in \cite{Bin,BPO,KMSY}.
However, all of them have technical limitations; for example, one cannot expect that general localizing invariants are representable there (since their categories are $\mbb{Z}$-linear), while they are representable in our category.
In addition to localizing invariants, we are pursing more general theories such as algebraic cobordism.
The initial algebra object in $\SH^\tr_\pbf(\Sch_S)$ should be regarded as the "periodic algebraic cobordism" $\PMGL$, which has not been constructed without $\mbb{A}^1$-localization.
In fact, we have a modification so that the initial algebra object represents the non-periodic algebraic cobordism $\MGL$.
Then Voevodsky's cobordism should be recovered as the $\mbb{A}^1$-localization $L_{\mbb{A}^1}\MGL$ and Annala's cobordism in \cite{Ann} should be recovered as $\pi_0\MGL$.
We hope to discuss these comparisons elsewhere.

We conclude this introduction with a brief outline of this paper.
In Section \ref{pbf}, we build setups for sheaves with transfers which satisfy projective bundle formula.
In Section \ref{euler}, we prove some basic properties of pbf-local sheaves with transfers in terms of Euler classes.
In Section \ref{vect}, we prove a key technical lemma for the proof of Theorem \ref{thm0:vect} and prove the comparison for $n=1$.
In Section \ref{chern}, we develop a theory of Chern classes and use it for computing the cohomology of grassmannians.
Then we complete the proof of Theorem \ref{thm0:vect} and get Theorem \ref{thm0:coh} as its corollary.
In Section \ref{loc}, we prove that every localizing invariant is representable in $\SH^\tr_\pbf(\Sch_S)$ so that a multiplicative localizing invariant yields an $\mbb{E}_\infty$-algebra in $\SH^\tr_\pbf(\Sch_S)$.

\subsubsection*{Convention}

We use the language of $\infty$-categories as set out in \cite{HTT,HA}.
We refer to \cite{SAG} for the theory of derived schemes.
We assume that all derived schemes are qcqs (quasi-compact and quasi-separated).
We say that a morphism $X\to Y$ of derived schemes is \textit{quasi-smooth} if it is of finite presentation and the cotangent complex $L_{X/Y}$ is of Tor-amplitude $\ge 1$, cf.\ \cite[\S2]{KR}.
A \textit{vector bundle} on a derived scheme $X$ is a locally free quasi-coherent module of finite rank on $X$.
Our treatment of sheaves with transfers are strongly influenced by the theory of framed correspondences as developed in \cite{EHKSYa}, from which we adopt some notations.

\section{Sheaves with transfers and projective bundle formula}\label{pbf}

Fix a qcqs derived scheme $S$.
Let $\Sch_S$ be the $\infty$-category of derived schemes of finite presentation over $S$ and $\Corr(\Sch_S)$ the $\infty$-category of correspondences/spans in $\Sch_S$, cf.\ \cite[Appendix C]{BH}.
By the assumption of finite presentation, the $\infty$-categories $\Sch_S$ and $\Corr(\Sch_S)$ are essentially small.
Recall that an object of $\Corr(\Sch_S)$ is an object of $\Sch_S$ and a morphism from $X$ to $Y$ in $\Corr(\Sch_S)$ is a diagram in $\Sch_S$
\[
\xymatrix{
	& Z \ar[ld] \ar[rd] & \\
	X & & Y.
}
\]
Let $\Corr^\fqsm(\Sch_S)$ be the subcategory of $\Corr(\Sch_S)$ spanned by morphisms whose left span is finite and quasi-smooth.

\begin{definition}\label{def:pst}
A \textit{presheaf with transfers} on $\Sch_S$ is a presheaf of spaces on $\Corr^\fqsm(\Sch_S)$ which preserves finite products.
We write $\PSh^\tr_\Sigma(\Sch_S)$ for the $\infty$-category $\PSh_\Sigma(\Corr^\fqsm(\Sch_S))$ of finite-product preserving presheaves.
\end{definition}

\begin{remark}
Since the $\infty$-category $\Corr^\fqsm(\Sch_S)$ is semiadditive by \cite[Lemma C.3]{BH}, presheaves with transfers take values in $\mbb{E}_\infty$-spaces by \cite[Proposition 2.3]{GGN}.
\end{remark}

\begin{notation}
Let $f\colon Y\to X$ be a morphism in $\Sch_S$ and $E$ a presheaf with transfers on $\Sch_S$.
We let ${}_!f=f\colon Y\to X$ denote the morphism in $\Corr^\fqsm(\Sch_S)$ and let $f^*:=E({}_!f)\colon E(X)\to E(Y)$.
If $f$ is finite quasi-smooth, then we let ${}^!f\colon X\to Y$ denote the dual morphism in $\Corr^\fqsm(\Sch_S)$ and let $f_*:=E({}^!f)\colon E(Y)\to E(X)$.
One can think of $f_*$ as a \textit{Gysin morphism}.
\end{notation}

\begin{example}
We write $\FQSm_S$ for the presheaf with transfers on $\Sch_S$ represented by $S$, which is by definition the \textit{moduli stack of finite quasi-smooth derived schemes}.
\end{example}

The canonical functor $\gamma\colon\Sch_S\to\Corr^\fqsm(\Sch_S)$ preserves finite coproducts and $\Corr^\fqsm(\Sch_S)$ is canonically endowed with a symmetric monoidal structure for which $\gamma$ is symmetric monoidal.
It follows that $\gamma$ induces an adjunction
\[
	\gamma^* \colon \PSh_\Sigma(\Sch_S) \rightleftarrows \PSh^\tr_\Sigma(\Sch_S) \colon \gamma_*
\]
and $\gamma^*$ is symmetric monoidal.
Here, we endow $\PSh_\Sigma(\Sch_S)$ and $\PSh^\tr_\Sigma(\Sch_S)$ with the unique symmetric monoidal structures for which the Yoneda embeddings are symmetric monoidal and the tensor products are compatible with small colimits, cf.\ \cite[Proposition 4.8.1.10]{HA}.
We note that the projection formula is automatic for presheaves with transfers.

\begin{lemma}[Projection formula]\label{lem:pf}
Let $E$ be a presheaf with transfers on $\Sch_S$ and $f\colon Y\to X$ a finite quasi-smooth morphism in $\Sch_S$.
Then the following diagrams commute
\[
\xymatrix@C-0.3pc{
	E(Y)\otimes E(X) \ar[r]^-{\id\otimes f^*} \ar[d]^{f_*\otimes\id} &
	E(Y)\otimes E(Y) \ar[r] &
	(E\otimes E)(Y) \ar[d]^{f_*} \\
	E(X)\otimes E(X) \ar[rr] & & (E\otimes E)(X)
}
\qquad
\xymatrix@C-0.3pc{
	E(X)\otimes E(Y) \ar[r]^-{f^*\otimes\id} \ar[d]^{\id\otimes f_*} &
	E(Y)\otimes E(Y) \ar[r] &
	(E\otimes E)(Y) \ar[d]^{f_*} \\
	E(X)\otimes E(X) \ar[rr] & & (E\otimes E)(X),
}
\]
where $E\otimes E$ denotes the tensor product in $\PSh^\tr_\Sigma(\Sch_S)$.
\end{lemma}
\begin{proof}
Since the tensor products commute with colimits, we may assume that $E$ is representable by a derived scheme in $\Sch_S$.
Then the commutativity is straightforward to check.
\end{proof}

\begin{definition}\label{def:sheaf}
A \textit{sheaf with transfers} on $\Sch_S$ is a presheaf $E$ with transfers such that $\gamma_*E$ is a sheaf on the Zariski site $\Sch_{S,\Zar}$.
We write $\Shv^\tr(\Sch_S)$ for the full subcategory of $\PSh^\tr_\Sigma(\Sch_S)$ spanned by sheaves with transfers.
\end{definition}

For each non-negative integer $n\ge 0$, let
\[
	\bar{\Delta}^n := \Proj\biggl(\frac{\mbb{Z}[U,T_0,\dotsc,T_n]}{U-\sum_{i=0}^n T_i}\biggr)
	\qquad
	\bar{\Delta}^n_\infty :=
	\begin{cases}
		V_+(U) & \text{if } n\ge 1 \\
		\varnothing & \text{if }n=0.
	\end{cases}
\]
Let $p\colon\bar{\Delta}^n\to\Spec(\mbb{Z})$ denote the projection and $\iota\colon\bar{\Delta}^n_\infty\to\bar{\Delta}^n$ denote the inclusion of the subscheme.
Since $\iota$ is quasi-smooth, it yields a morphism ${}^!\iota\colon\bar{\Delta}^n\to\bar{\Delta}^n_\infty$ in $\Corr^\fqsm(\Sch)$.

\begin{definition}\label{def:pbf}
We say that a (pre)sheaf $E$ with transfers on $\Sch_S$ \textit{satisfies projective bundle formula} or \textit{is pbf-local} if, for every $X\in\Sch_S$ and $n\ge 0$, the map
\[
	p^*\oplus\iota_* \colon E(X)\oplus E(X\times\bar{\Delta}^n_\infty) \to E(X\times\bar{\Delta}^n)
\]
is an equivalence.
We write $\Shv^\tr_\pbf(\Sch_S)$ for the full subcategory of $\PSh^\tr_\Sigma(\Sch_S)$ spanned by pbf-local sheaves with transfers.
\end{definition}

\begin{remark}
If $E$ is a pbf-local presheaf with transfers, then for any $X$-points $a,b\colon X\rightrightarrows\bar{\Delta}^n_X$ not meeting $\bar{\Delta}^n_{\infty,X}$ the two maps $a^*,b^*\colon E(\bar{\Delta}^n_X)\rightrightarrows E(X)$ are homotopic to each other.
\end{remark}

\begin{remark}
Every pbf-local sheaf with transfers has its values in grouplike $\mbb{E}_\infty$-spaces.
We give a proof in the next section, cf.\ Corollary \ref{cor:fund}.
\end{remark}

The $\infty$-category $\Shv^\tr_\pbf(\Sch_S)$ is an accessible localization of $\PSh^\tr_\Sigma(\Sch_S)$ by \cite[Proposition 5.5.4.15]{HTT}.
We denote the localization functor by
\[
	L_\mot\colon\PSh^\tr_\Sigma(\Sch_S) \to \Shv^\tr_\pbf(\Sch_S).
\]
Then $\Shv^\tr_\pbf(\Sch_S)$ admits a unique symmetric monoidal structure for which the localization functor $L_\mot$ is symmetric monoidal by \cite[Proposition 4.1.7.4]{HA}.

\begin{definition}\label{def:sh}
A \textit{pbf-local sheaf of spectra with transfers} on $\Sch_S$ is a spectrum object in the $\infty$-category $\Shv^\tr_\pbf(\Sch_S)$ in the sense of \cite[\S1.4.2]{HA}.
We write $\SH^\tr_\pbf(\Sch_S)$ for the $\infty$-category of spectrum objects in $\Shv^\tr_\pbf(\Sch_S)$.
\end{definition}

\begin{remark}
A pbf-local sheaf $E$ of spectra with transfers can be identified with a presheaf of spectra on $\Corr^\fqsm(\Sch_S)$ such that $\Omega^{\infty-n}E$ is a pbf-local sheaf with transfers for every $n\ge 0$.
\end{remark}

There is an adjunction
\[
	B^\infty_\mot \colon \Shv^\tr_\pbf(\Sch_S) \rightleftarrows \SH^\tr_\pbf(\Sch_S) \colon \Omega^\infty
\]
and the left adjoint $B^\infty_\mot$ is called the \textit{infinite bar construction}.\footnote{In fact, this is an adjoint equivalence, as proved in \cite[Theorem 2.4.5]{AI}, but we will not use it in this paper.}
Then $\SH^\tr_\pbf(\Sch_S)$ admits a unique symmetric monoidal structure for which the infinite bar construction is symmetric monoidal by \cite[Theorem 5.1]{GGN}.

\begin{example}
Every localizing invariant in the sense of \cite{BGT} naturally defines a pbf-local sheaf of spectra with transfers on $\Sch_S$ for each qcqs derived scheme $S$.
Moreover, if a localizing invariant is multiplicative, then it defines an $\mbb{E}_\infty$-algebra in $\SH^\tr_\pbf(\Sch_S)$.
We give a proof in Section \ref{loc}.
\end{example}

\subsubsection*{Base changes}

Let $f\colon T\to S$ be a morphism of derived schemes.
Then the base change functor $\Sch_S\to \Sch_T$ induces adjunctions
\begin{align*}
	f^*\colon \PSh^\tr_\Sigma(\Sch_S) &\rightleftarrows \PSh^\tr_\Sigma(\Sch_T) \colon f_* \\
	f^*\colon \Shv^\tr_\pbf(\Sch_S) &\rightleftarrows \Shv^\tr_\pbf(\Sch_T) \colon f_* \\
	f^*\colon \SH^\tr_\pbf(\Sch_S) &\rightleftarrows \SH^\tr_\pbf(\Sch_T) \colon f_*.
\end{align*}
It is straightforward to see that the left adjoints $f^*$ are symmetric monoidal and compatible with the localization functor $L_\mot$ and the infinite bar construction $B^\infty_\mot$.

\subsubsection*{Approximation of the motivic localization}

The assignments $n\mapsto\bar{\Delta}^n$ and $n\mapsto\bar{\Delta}^n_\infty$ form semi-cosimplicial objects in the category of schemes in a standard way, and the map $\iota\colon\bar{\Delta}^n_\infty\to\bar{\Delta}^n$ is assembled into a morphism of semi-cosimplicial schemes $\iota\colon\bar{\Delta}^\bullet_\infty\to\bar{\Delta}^\bullet$.
Since the square
\[
\xymatrix{
	\bar{\Delta}^m_\infty \ar[r]^\iota & \bar{\Delta}^m \\
	\bar{\Delta}^n_\infty \ar[r]^\iota \ar[u]^\theta & \bar{\Delta}^n \ar[u]^\theta
}
\]
is cartesian in $\Sch$ for each injection $\theta\colon[n]\hookrightarrow[m]$, the map $\iota$ yields a morphism of semi-cosimplicial objects ${}^!\iota\colon \bar{\Delta}^\bullet\to\bar{\Delta}^\bullet_\infty$ in $\Corr^\fqsm(\Sch)$.

\begin{lemma}\label{lem:approx}
Let $E$ be a presheaf with transfers.
Then the canonical map $E\to L_\mot E$ factors as
\[
	E
	\to \lvert E(\bar{\Delta}^\bullet)\rvert
	\to \lvert \cofib(E(\bar{\Delta}^\bullet_\infty) \xrightarrow{\iota_*} E(\bar{\Delta}^\bullet)) \rvert
	\to L_\mot E,
\]
where $\lvert(-)\rvert$ denotes the geometric realization.
\end{lemma}
\begin{proof}
Note that if $E$ is pbf-local then there is a canonical equivalence
\[
	E(X) \simeq \lvert \cofib(E(\bar{\Delta}^\bullet_\infty\times X) \xrightarrow{\iota_*} E(\bar{\Delta}^\bullet\times X)) \rvert,
\]
from which the desired factorization is immediate.
\end{proof}

\section{Euler classes}\label{euler}

\begin{definition}\label{def:pmgl}
We define $\PMGL_S:=L_\mot\FQSm_S$, which is a unit object in the symmetric monoidal $\infty$-category $\Shv^\tr_\pbf(\Sch_S)$.
\end{definition}

\begin{remark}
$\PMGL_S$ is a version of the periodic algebraic cobordism.
Comparison with the existing theories will be discussed elsewhere.
See \cite[Corollary 3.4.2]{EHKSYb} for the similar description of the $\mbb{A}^1$-local algebraic cobordism.
\end{remark}

\begin{definition}\label{def:euler}
Let $\mcal{E}$ be a vector bundle on a derived scheme $X\in\Sch_S$.
We define the \textit{Euler class} $e(\mcal{E})\in\pi_0\PMGL_S(X)$ by $e(\mcal{E}):=s^*s_*(\mbf{1}_X)$, where $s$ denotes the zero section of the total space of $\mcal{E}$
\[
	s\colon X \to \mbb{V}(\mcal{E}):=\Spec(\Sym(\mcal{E}^\vee)).
\]
\end{definition}

\begin{remark}
The Euler classes are compatible with base changes, i.e., for a morphism $f\colon Y\to X$ in $\Sch_S$, we have $f^*e(\mcal{E})=e(f^*\mcal{E})$.
Since every pbf-local sheaf $E$ with transfers is a module over $\PMGL_S$, the Euler class of a vector bundle $\mcal{E}$ on $X$ yields an endomorphism of $E(X)$ up to homotopies.
\end{remark}

Let $a$ be a global section of a vector bundle $\mcal{E}$ on $X$.
Then the \textit{derived vanishing locus} $V_a$ of $a$ is a derived scheme defined by the cartesian square
\[
\xymatrix{
	V_a \ar[r]^-{j_a} \ar[d] & X \ar[d]^s \\
	X \ar[r]^-a & \mbb{V}(\mcal{E}).
}
\]
By definition, we have $e(\mcal{E})=j_{0*}(\mbf{1}_{V_0})$.

\begin{lemma}\label{lem:euler}
Let $\mcal{E}$ be a vector bundle on a derived scheme $X\in\Sch_S$.
Then, for any global section $a$ of $\mcal{E}$, we have $j_{a*}(\mbf{1}_{V_a})=e(\mcal{E})$.
In particular, if $\mcal{E}$ admits a nowhere vanishing global section, then $e(\mcal{E})=0$.
\end{lemma}
\begin{proof}
Let $V$ be the derived vanishing locus of the global section $at_0$ of the twisting sheaf $\mcal{E}(1)$ on $\mbb{P}^1_X$, and let $j\colon V\to\mbb{P}^1_X$ be the inclusion.
Then $i_0^*V=V_0$ and $i_\infty^*V=V_a$.
Therefore, $i_0^*j_*(\mbf{1}_V)=j_{0*}(\mbf{1}_{V_0})$ and $i_\infty^*j_*(\mbf{1}_V)=j_{a*}(\mbf{1}_{V_a})$.
We conclude the proof by noting $i_0^*=i_\infty^*$.
\end{proof}

\begin{lemma}\label{lem:pbf}
Let $E$ be a pbf-local sheaf with transfers on $\Sch_S$.
Let $\mcal{E}$ be a vector bundle of rank $r\ge 1$ on a derived scheme $X\in\Sch_S$ and $\xi$ the Euler class of the canonical line bundle $\mcal{O}(1)$ on $\mbb{P}(\mcal{E})$.
Then the morphism
\[
	\sum_{i=0}^{r-1} (\xi^i\cdot p^*) \colon \bigoplus_{i=0}^{r-1} E(X) \to E(\mbb{P}(\mcal{E}))
\]
is an equivalence.
\end{lemma}
\begin{proof}
Since the question is local on $X$, we may assume that the $\mcal{E}$ is trivial, i.e., $\mcal{E}=\mcal{O}_X^r$.
We prove by induction on $r$.
The case $r=1$ is trivial, and let $r\ge 2$.
Fix a linear embedding $j\colon\mbb{P}^{r-2}_X\to\mbb{P}^{r-1}_X$, which is given by $x\in H^0(\mbb{P}^{r-1}_X,\mcal{O}(1))$.
Then we have a cartesian diagram
\[
\xymatrix{
	\mbb{P}^{r-2}_X \ar[r]^-j \ar[d] & \mbb{P}^{r-1}_X \ar[d]^s \\
	\mbb{P}^{r-1}_X \ar[r]^-x & \mbb{V}(\mcal{O}(1))
}
\]
and $j_*(\mbf{1})=\xi$ by Lemma \ref{lem:euler}.
It follows that the diagram
\[
\xymatrix@C+2pc{
	\bigoplus_{i=0}^{r-1} E(X) \ar[r]^-{\sum_{i=0}^{r-1}(\xi^i\cdot p^*)} & E(\mbb{P}^{r-1}_X) \\
	\bigoplus_{i=0}^{r-1} E(X) \ar[r]^-{1 \oplus \sum_{i=0}^{r-2}(\xi^i\cdot p^*)} \ar@{=}[u] & E(X) \oplus E(\mbb{P}^{r-2}_X) \ar[u]_{p^*\oplus j_*}
}
\]
commutes.
The right vertical arrow is an equivalence since $E$ is pbf-local.
Therefore, the result follows from the induction hypothesis.
\end{proof}

\begin{corollary}\label{cor:pbf}
Let $E$ be a homotopy commutative algebra in $\SH^\tr_\pbf(\Sch_S)$, i.e., a commutative algebra object in the homotopy category.
Then there is a canonical ring isomorphism
\[
	\pi_*E(\mbb{P}^n_X) \simeq \pi_*E(X)[\xi]/\xi^{n+1}
\]
for every $X\in\Sch_S$ and $n\ge 0$.
\end{corollary}
\begin{proof}
By Lemma \ref{lem:pbf}, it remains to show that $\xi^{n+1}=0$ in $\pi_0\PMGL_S(\mbb{P}^n_X)$.
By the projection formula (Lemma \ref{lem:pf}), $\xi^{n+1}$ is the image of the unit by the Gysin morphism along the inclusion of the derived intersection of $(n+1)$-copies of a hyperplane in $\mbb{P}^n$, but this is empty and thus $\xi^{n+1}=0$.
\end{proof}

We prove that every pbf-local sheaf with transfers satisfies Bass fundamental theorem, cf.\ \cite[\S6]{TT}.
Consider the affine cover 
\[
	\{D_+(T_0),D_+(T_1)\}\to\bar{\Delta}^1.
\]
The Euler class $e(\mcal{O}(1))\in\pi_0\PMGL_S(\bar{\Delta}^1_X)$ is sent to zero in $\pi_0\PMGL_S(D_+(T_0)_X)$ and $\pi_0\PMGL_S(D_+(T_1)_X)$ by Lemma \ref{lem:euler}.
Therefore, it lifts to $\nu\in\pi_1\PMGL_S(D_+(T_0T_1)_X)$.

\begin{lemma}[Bass fundamental theorem]\label{lem:fund}
Let $E$ be a pbf-local sheaf with transfers on $\Sch_S$.
Then, for every $X\in\Sch_S$ and $i\ge 0$, there is a split exact sequence
\[
	0 \to \pi_{i+1}E(X) \to \pi_{i+1}E(D_+(T_0)_X) \oplus \pi_{i+1}E(D_+(T_1)_X) \to \pi_{i+1}E(D_+(T_0T_1)_X) \to \pi_iE(X) \to 0,
\]
where the multiplication by $\nu$ gives a splitting $\pi_iE(X)\to\pi_{i+1}E(D_+(T_0T_1)_X)$.
\end{lemma}
\begin{proof}
Consider the diagram
\[
\xymatrix{
	& \pi_{i+1}E(D_+(T_0)_X)\oplus\pi_{i+1}E(D_+(T_1)_X) \ar[d] & \\
	& \pi_{i+1}E(D_+(T_0T_1)_X) \ar[d]^\partial \ar[ld] & \\
	\pi_iE(X) \ar[r]_-{e(\mcal{O}(1))\cdot p^*} \ar@{.>}@/^1pc/[ru] & \pi_iE(\bar{\Delta}^1_X) \ar[d] & \pi_iE(X) \ar[ld] \ar[l]_-{p^*} & \\
	& \pi_iE(D_+(T_0)_X)\oplus\pi_iE(D_+(T_1)_X). &
}
\]
The vertical sequence is exact since $E$ is a sheaf and the horizontal sequence exhibits $\pi_iE(\bar{\Delta}^1_X)$ as a direct sum of two copies of $\pi_iE(X)$ by Lemma \ref{lem:pbf}.
The boundary map $\partial$ factors through the left summand as indicated since the right diagonal map is injective.
Now the result follows from a simple diagram chase.
\end{proof}

\begin{corollary}\label{cor:fund}
Every pbf-local sheaf with transfers on $\Sch_S$ has its values in grouplike $\mbb{E}_\infty$-spaces.
\end{corollary}
\begin{proof}
It suffices to show that if $E$ is a pbf-local sheaf with transfers on $\Sch_S$ then $\pi_0E(X)$ is a group for every $X\in\Sch_S$.
Since $\pi_0E(X)$ is a quotient of $\pi_1E(D_+(T_0T_1)_X)$ by Lemma \ref{lem:fund}, it is a group.
\end{proof}

\section{Moduli stack of vector bundles}\label{vect}

In this section, we study the homotopy type of the moduli stack of vector bundles.
We start by recalling the construction of the moduli stack of vector bundles.
Let $\Sch$ be the $\infty$-category of all derived schemes.
Then there is a functor $\QCoh\colon\Sch^\op\to\Cat_\infty$ classifying all quasi-coherent modules, which is constructed as in \cite[\S6.2.2]{SAG}, where $\Cat_\infty$ denotes the $\infty$-category of possibly large $\infty$-categories.
For a non-negative integer $n$, let $\QCoh^\mrm{lfree}_n$ be the subfunctor of $\QCoh$ spanned by locally free quasi-coherent modules of rank $n$.
Then the \textit{moduli stack $\Vect_n$ of vector bundles of rank $n$} is the functor defined as the composite
\[
	\Vect_n \colon\: \Sch^\op \xrightarrow{\QCoh^\mrm{lfree}_n} \Cat_\infty \xrightarrow{(-)^\sim} \mscr{S},
\]
where $(-)^\sim$ is the functor taking the maximal subgroupoids of $\infty$-categories.
We write $\Pic:=\Vect_1$, which is by definition the \textit{Picard stack}.

The presheaf $\Vect_n$ is a fpqc sheaf since so is the presheaf $\mrm{QCoh}$ classifying quasi-coherent modules by \cite[Proposition 6.2.3.1]{SAG} and the property being locally free of rank $n$ is local for the fpqc topology by \cite[Proposition 2.9.1.4]{SAG}.
The presheaf $\Vect_n$ is finitary in the sense that it carries filtered limits of qcqs derived schemes with affine transition maps to colimits by \cite[Corollary 4.5.1.10]{SAG}.
For a qcqs derived scheme $S$, we write $\Vect_{n,S}$ for the restriction of $\Vect_n$ to $\Sch_S$.
Then $\Vect_{n,S}$ is compatible with base changes, i.e., for every morphism $f\colon T\to S$ of derived schemes, the map $f^*\Vect_{n,S}\to\Vect_{n,T}$ is a Zariski local equivalence, cf. \cite[Proposition A.0.4]{EHKSYb}

For non-negative integers $n$ and $N$, the $n$-th grassmannian $\Gr_n(\mcal{O}^N)$ of $\mcal{O}^N$ classifies all quotients $\mcal{O}^N\twoheadrightarrow\mcal{E}$, where $\mcal{E}$ is a vector bundle of rank $n$.
The projection $\mcal{O}^{N+1}\to \mcal{O}^N$ discarding the last factor induces an immersion $\Gr_n(\mcal{O}^N)\hookrightarrow\Gr_n(\mcal{O}^{N+1})$.
Let $\Gr_n:=\colim_N\Gr_n(\mcal{O}^N)$ and regard it as an ind-scheme.
We write $\mbb{P}^\infty:=\Gr_1$, which is the infinite projective space.

\begin{theorem}\label{thm:vect}
For every qcqs derived scheme $S$ and $n\ge 0$, the canonical map
\[
	L_\mot\gamma^*\Gr_{n,S} \to L_\mot\gamma^*\Vect_{n,S}
\]
is an equivalence.
\end{theorem}

The proof is completed in the next section.
In this section, we prove a key technical lemma (Lemma \ref{lem:key}) and prove the equivalence for $n=1$.
Note that we may assume that $S=\Spec(\mbb{Z})$ to prove Theorem \ref{thm:vect} since both sides commute with base changes.
In the rest of this section, we work over the $\infty$-category $\Sch_\mbb{Z}:=\Sch_{\Spec(\mbb{Z})}$ unless otherwise stated.
\begin{lemma}\label{lem:pic}
The canonical map
\[
	L_\mot \gamma^*\mbb{P}^\infty \to L_\mot \gamma^*\Pic
\]
admits a left inverse.
\end{lemma}
\begin{proof}
Consider the universal line bundle $\mcal{L}_\univ$ on $\Pic$.
Then the total space $\mbb{V}(\mcal{L}_\univ)$ is defined as the stack classifying line bundles with a global section, i.e., $\mbb{V}(\mcal{L}_\univ)(X)$ is the space of all maps $\mcal{O}_X\to\mcal{L}$ with $\mcal{L}$ being a line bundle on $X$.
We have a map $s\colon\Pic\to\mbb{V}(\mcal{L}_\univ)$ classifying zero sections, which can be expressed as a colimit of quasi-smooth closed immersions.
Hence, the Gysin morphism $s_*$ is well-defined and the Euler class of $\mcal{L}_\univ$ is defined by $\xi:=s^*s_*(\mbf{1})\in \pi_0\PMGL(\Pic)$.
We note that $\PMGL(\Pic)$ is complete with respect to $\xi$ since it is a limit of $\PMGL(\mbb{G}_m^{\times n})$ where $\xi=0$ and a limit of complete modules is complete.
Then we have a commutative diagram
\[
\xymatrix{
	\prod_{i=0}^\infty (L_\mot\gamma^*\mbb{P}^\infty)(\Spec(\mbb{Z})) \ar[d]_{\sum_{i=0}^\infty (\xi^i\cdot p^*)} \ar[rd]^\simeq & \\ 
	(L_\mot\gamma^*\mbb{P}^\infty)(\Pic) \ar[r] &
	(L_\mot\gamma^*\mbb{P}^\infty)(\mbb{P}^\infty).
}
\]
The left vertical map is well-defined because of the $\xi$-completeness of $\PMGL(\Pic)$.
The diagonal arrow is an equivalence by Lemma \ref{lem:pbf}.
In particular, the canonical map $\gamma^*\mbb{P}^\infty\to L_\mot\gamma^*\mbb{P}^\infty$ lifts to a map $\gamma^*\Pic\to L_\mot\gamma^*\mbb{P}^\infty$ up to homotopies, which gives a desired left inverse.
\end{proof}

The next goal is to construct a right inverse of $L_\mot\gamma^*\Gr_n\to L_\mot\gamma^*\Vect_n$.
In order to do that, we pursue the idea of \textit{closed gluing} as in \cite[\S{}A.2]{EHKSYa}, which provides a method to show that some map is an $\mbb{A}^1$-homotopy equivalence.
See \cite[Proposition 4.7]{HJNY}, where it is used to show that the canonical map $\Gr_n\to\Vect_n$ is an $\mbb{A}^1$-homotopy equivalence on affine schemes.
Adopting this to our situation, we can try to solve the lifting problem
\[
\xymatrix{
	& \lvert \Gr_n(\bar{\Delta}^\bullet) \rvert \ar[d] \\
	\Vect_n \ar[r]^-{p^*} \ar@{.>}@/^1pc/[ru]^-{\exists?} & \lvert \Vect_n(\bar{\Delta}^\bullet) \rvert.
}
\]
However, it will turn out soon that this is impossible, mainly because the sheaf cohomology of $\bar{\Delta}^*$ is non-trivial.
To fix this, note that for each $l\ge 1$ the twisting sheaves $\mcal{O}(l)$ on $\bar{\Delta}^*$ define a point of the semi-simplicial sheaf $\Pic(\bar{\Delta}^\bullet)$.
Then we prove the following.

\begin{lemma}\label{lem:key}
Let $k$ be a non-negative integer.
\begin{enumerate}[label=\upshape{(\Alph*)},leftmargin=*]
\item The diagram
\[
\xymatrix@C+1.5pc{
	& \lvert \cosk_k(\Gr_n(\bar{\Delta}^\bullet)) \rvert \ar[d] \\
	\Vect_n \ar[r]^-{\mcal{O}(k+1)\otimes p^*} \ar@{.>}@/^1.5pc/[ru] & \lvert \cosk_k(\Vect_n(\bar{\Delta}^\bullet)) \rvert
}
\]
admits a lift Zariski locally as indicated.
\item The composite
\[
	\Vect_n
	\xrightarrow{\mcal{O}(k+1)\otimes p^*} \lvert \Vect_n(\bar{\Delta}^\bullet) \rvert
	\to \gamma_*L_\mot\gamma^*\Vect_n
\]
is homotopic to the canonical map on its finite truncations, where the second arrow is the map in Lemma \ref{lem:approx}.
\end{enumerate}
\end{lemma}
\begin{proof}
Note that the $\infty$-category $\Shv(\Sch_\mbb{Z})$ of Zariski sheaves is a hypercomplete $\infty$-topos since each derived scheme $X\in\Sch_\mbb{Z}$ is assumed to be of finite presentation over $\mbb{Z}$.

(A) Let $F$ be the semi-simplicial object in $\Shv(\Sch_\mbb{Z})$ defined by the pullback square
\[
\xymatrix@C+1.5pc{
	F \ar[r] \ar[d] & \cosk_k(\Gr_n(\bar{\Delta}^\bullet)) \ar[d] \\
	\Vect_n \ar[r]^-{\mcal{O}(k+1)\otimes p^*} & \cosk_k(\Vect_n(\bar{\Delta}^\bullet)).
}
\]
We show that the induced map $\lvert F\rvert\to\lvert\Vect_n\rvert$ on the geometric realization is an equivalence in $\Shv(\Sch_\mbb{Z})$.
By \cite[Theorem A.5.3.1]{SAG}, it suffices to show that the map $F\to\Vect_n$ is a trivial Kan fibration, i.e., the map
\[
	F[\Delta^m] \to \Vect_n[\Delta^m]\times_{\Vect_n[\partial\Delta^m]}F[\partial\Delta^m]
\]
is an effective epimorphism for every $m\ge 0$, see [op.\ cit., A.5.1.6] for the notation.
This map is identified with the map
\[
	\Vect_n\times_{\cosk_k(\Vect_n(\bar{\Delta}^\bullet))[\Delta^m]}\cosk_k(\Gr_n(\bar{\Delta}^\bullet))[\Delta^m]
	\to \Vect_n\times_{\cosk_k(\Vect_n(\bar{\Delta}^\bullet))[\partial\Delta^m]}\cosk_k(\Gr_n(\bar{\Delta}^\bullet))[\partial\Delta^m]
\]
and it is an equivalence for $m>k$, and thus we may assume that $m\le k$.
Then, since $\Gr_n$ and $\Vect_n$ satisfy closed gluing by \cite[Example 17.3.1.3]{SAG}, the map is further identified with
\[
	\Vect_n\times_{\Vect_n(\bar{\Delta}^m)}\Gr_n(\bar{\Delta}^m)
	\to \Vect_n\times_{\Vect_n(\partial\bar{\Delta}^m)}\Gr_n(\partial\bar{\Delta}^m)
\]
by \cite[Lemma A.2.6]{EHKSYa}.
Over an animated ring $A$ such that $\pi_0(A)$ is local, $\pi_0$ of the right hand side is identified with the set of homotopy equivalent classes of surjections $\mcal{O}^{\oplus\infty}_{\partial\bar{\Delta}^m_A}\to\mcal{O}(k+1)^{\oplus n}\vert_{\partial\bar{\Delta}^m_A}$.
We have to show that such a surjection lifts to a surjection $\mcal{O}^{\oplus\infty}_{\bar{\Delta}^m_A}\to \mcal{O}(k+1)^{\oplus n}$ up to homotopies.
Consider the diagram
\[
\xymatrix{
	& \mcal{O}^{\oplus\infty}_{\bar{\Delta}^m_A} \ar[r] \ar@{.>}[d]^{\alpha'} & \mcal{O}^{\oplus\infty}_{\partial\bar{\Delta}^m_A} \ar[d]^\alpha \\
	\mcal{O}(k-m)^{\oplus n} \ar[r] & \mcal{O}(k+1)^{\oplus n} \ar[r] & \mcal{O}(k+1)^{\oplus n}\vert_{\partial\bar{\Delta}^m_A}
}
\]
for a given surjection $\alpha$.
Note that the bottom row is a fiber sequence.
Since $H^1(\bar{\Delta}^m_A,\mcal{O}(k-m)^{\oplus n})=0$, there exists a lift $\alpha'$ as indicated.
Furthermore, since $\mcal{O}(k-m)$ is globally generated, we can add extra sections to ensure that the lift $\alpha'$ is surjective.
This completes the proof of (A).

(B) For each $m\ge 0$, we have a commutative diagram of semi-simplicial objects 
\[
\xymatrix@C-0.5pc{
	\Vect_n\otimes\Delta^m \ar[d]^\simeq \ar[rr]^-{\mcal{O}(k+1)\otimes p^*} & &
	\Vect_n(\bar{\Delta}^m)\otimes\Delta^m \ar[d] \ar[r] &
	\displaystyle \frac{\gamma_*L_\mot\gamma^*\Vect_n(\bar{\Delta}^m)}{\gamma_*L_\mot\gamma^*\Vect_n(\bar{\Delta}^m_\infty)}\otimes\Delta^m \ar[d] &
	\gamma_*L_\mot\gamma^*\Vect_n\otimes\Delta^m \ar[l]_-{p^*}^-\simeq \ar[d] \\
	\Vect_n \ar[rr]^-{\mcal{O}(k+1)\otimes p^*} & &
	\cosk_m(\Vect_n(\bar{\Delta}^\bullet)) \ar[r] & 
	\displaystyle \cosk_m\biggl(\frac{\gamma_*L_\mot\gamma^*\Vect_n(\bar{\Delta}^\bullet)}{\gamma_*L_\mot\gamma^*\Vect_n(\bar{\Delta}^\bullet_\infty)}\biggr) &
	\cosk_m(\gamma_*L_\mot\gamma^*\Vect_n). \ar[l]_-{p^*}^-\simeq
}
\]
Hence, it suffices to show that the top horizontal composition is homotopic to the canonical map.
Note that, for any point $*\in\bar{\Delta}^m$ not meeting $\bar{\Delta}^m_\infty$, the pullback along the inclusion $i\colon\{*\}\to\bar{\Delta}^m$ induces an inverse of $p^*$ on the top.
Since $i^*\mcal{O}(k+1)$ is trivial, the assertion follows.
\end{proof}

\begin{corollary}\label{cor:key}
For every $n,k\ge 0$, the diagram 
\[
\xymatrix{
	& (\gamma_*L_\mot \gamma^*\Gr_n)_{\le k} \ar[d] \\
	\Vect_n \ar[r] \ar@{.>}@/^1pc/[ru] & (\gamma_*L_\mot \gamma^*\Vect_n)_{\le k}
}
\]
admits a lift Zariski locally as indicated, where $(-)_{\le k}$ denotes the pointwise $k$-truncation.
\end{corollary}
\begin{proof}
By Lemma \ref{lem:key}, we have a commutative diagram with a Zariski local lift as indicated
\[
\xymatrix{
	& &  \lvert\cosk_k(\Gr_n(\bar{\Delta}^\bullet))\rvert \ar[d] \ar[r] &
	\lvert \cosk_k(\gamma_*L_\mot\gamma^*\Gr_n) \rvert \ar[d] \ar[r] &
	(\gamma_*L_\mot\gamma^*\Gr_n)_{\le k} \ar[d] \\
	\Vect_n \ar@{.>}@/^1.5pc/[rru] \ar[rr]^-{\mcal{O}(k+1)\otimes p^*} \ar@/_1.5pc/[rrrr]_-{\mrm{can}} & &
	\lvert\cosk_k(\Vect_n(\bar{\Delta}^\bullet))\rvert \ar[r] & 
	\lvert \cosk_k(\gamma_*L_\mot\gamma^*\Vect_n) \rvert \ar[r] &
	(\gamma_*L_\mot\gamma^*\Vect_n)_{\le k},
}
\]
where the rightmost horizontal arrows come from the canonical map $\lvert\cosk_kX\rvert\to\lvert\cosk_kX\rvert_{\le k}\simeq\lvert X\rvert_{\le k}$ for a semi-simplicial space $X$. 
Hence, we get a desired lift.
\end{proof}

\begin{corollary}\label{cor:pic}
For every qcqs derived scheme $S$, the canonical map
\[
	L_\mot \gamma^*\mbb{P}^\infty_S \to L_\mot \gamma^*\Pic_S
\]
is an equivalence.
\end{corollary}
\begin{proof}
We may assume $S=\Spec(\mbb{Z})$.
By Corollary \ref{cor:key}, we have a commutative diagram
\[
\xymatrix{
	& L_\Zar(\gamma_*L_\mot \gamma^*\mbb{P}^\infty)_{\le k} \ar[d] \\
	\Pic \ar[r] \ar@/^1pc/[ru]^-{\phi_k} & L_\Zar(\gamma_*L_\mot \gamma^*\Pic)_{\le k}.
}
\]
The map $\phi_k$ is characterized as a unique map which makes the diagram commutative since the right vertical map has a left inverse by Lemma \ref{lem:pic}.
In particular, these maps are assembled into a map
\[
	\phi \colon \Pic \to \lim_k L_\Zar(\gamma_*L_\mot\gamma^*\mbb{P}^\infty)_{\le k}.
\]
Since the $\infty$-topos $\Shv(X_\Zar)$ has finite homotopy dimension for each derived scheme $X\in\Sch_\mbb{Z}$, we have an equivalence $\lim_kL_\Zar(\gamma_*L_\mot\gamma^*\mbb{P}^\infty)_{\le k}\simeq\gamma_*L_\mot\gamma^*\mbb{P}^\infty$.
Therefore, we get a morphism $\Pic\to\gamma_*L_\mot\gamma^*\mbb{P}^\infty$ inducing a right inverse of the canonical map $L_\mot\gamma^*\mbb{P}^\infty_S\to L_\mot\gamma^*\Pic_S$.\footnote{It is not clear if the obtained map is indeed a right inverse, since $\phi_k$ above is unique only up to non-canonical homotopies. This gap has been fixed in \cite{AI}, and the result here remains true.}
Since the canonical map has a left inverse by Lemma \ref{lem:pic}, we conclude that it is an equivalence.
\end{proof}

\begin{corollary}\label{cor:pic_coh}
Let $E$ be a homotopy commutative algebra in $\SH^\tr_\pbf(\Sch_S)$.
Then there is a canonical ring isomorphism
\[
	\pi_*E(\Pic_S) \simeq \pi_*E(S)[[t]],
\]
where $t$ is the Euler class of the universal line bundle.
\end{corollary}
\begin{proof}
This is immediate from Corollary \ref{cor:pbf} and Corollary \ref{cor:pic}.
\end{proof}

\section{Chern classes and formal group laws}\label{chern}

\begin{definition}\label{def:chern}
Let $\mcal{E}$ be a vector bundle of rank $r\ge 1$ on a derived scheme $X\in\Sch_S$.
We let $c_0(\mcal{E})=1$.
For $1\le i\le r$, we define the \textit{$i$-th Chern class} $c_i(\mcal{E})\in\pi_0\PMGL_S(X)$ to be unique elements which satisfy the formula
\[
	\sum_{i=0}^r (-1)^i\xi^i\cdot p^*c_{r-i}(\mcal{E}) = 0
\]
in $\pi_0\PMGL_S(\mbb{P}(\mcal{E}))$, cf.\ Lemma \ref{lem:pbf}.
We write $c(\mcal{E}):=\sum_{i=0}^r c_i(\mcal{E})t^i$, which is the \textit{total Chern class}.
\end{definition}

\begin{remark}
The Chern classes are compatible with base changes since the Euler classes are.
If $\mcal{L}$ is a line bundle, then $c_1(\mcal{L})=e(\mcal{L})$.
\end{remark}

Let $m\colon\Pic\times\Pic\to\Pic$ be the map classifying the tensor products of line bundles.
Consider the induced map
\[
	m^*\colon \PMGL_S(\Pic_S) \to \PMGL_S(\Pic_S\times\Pic_S),
\]
and let $f_\univ$ be the image of $c_1(\mcal{L}_\univ)$ in
\[
	\pi_0\PMGL_S(\Pic_S\times\Pic_S) \simeq \pi_0\PMGL_S(S)[[x,y]],
\]
where the isomorphism is by Corollary \ref{cor:pic_coh}.
Then $f_\univ$ is a formal group law over $\pi_0\PMGL_S(S)$.
Moreover, the formal inverse power series of $t$ in $\pi_0\PMGL_S(S)[[t]]\simeq\pi_0\PMGL_S(\Pic_S)$ corresponds to the first Chern class of $\mcal{L}_\univ^{-1}$.
The next two lemmas are the standard properties of Chern classes and formal group laws.

\begin{lemma}\label{lem:fgl}
Suppose that $X\in\Sch_S$ is noetherian and admits an ample line bundle.
Then the first Chern classes of line bundles on $X$ are nilpotent in $\pi_0\PMGL_S(X)$ and we have an equality
\[
	c_1(\mcal{L}\otimes\mcal{L}')=f_\univ(c_1(\mcal{L}),c_1(\mcal{L}')) \in \pi_0\PMGL_S(X)
\]
for every pair of line bundles $\mcal{L},\mcal{L}'$ on $X$.
\end{lemma}
\begin{proof}
If $\mcal{L}$ is a globally generated line bundle, then it is generated by a finite number of global sections since $X$ is noetherian.
Then it is immediate from the definition of the Euler class that $e(\mcal{L})=c_1(\mcal{L})$ is nilpotent.
It follows that the map
\[
	\pi_0\PMGL_S(X)[[t]] \simeq \pi_0\PMGL(\Pic_S) \to \pi_0\PMGL_S(X)
\]
induced by the map $X\to\Pic$ classifying $\mcal{L}$ factors through $\pi_0\PMGL_S(X)[t]/t^m$ for some $m>0$.
In particular, $c_1(\mcal{L}^{-1})$, which is the image of the formal inverse of $t$, is also nilpotent.

Since $X$ admits an ample line bundle, every line bundle $\mcal{L}$ on $X$ can be written as $\mcal{L}_1\otimes\mcal{L}_2^{-1}$ for some globally generated line bundles $\mcal{L}_1$ and $\mcal{L}_2$.
Since we have seen that $c_1(\mcal{L}_1)$ and $c_1(\mcal{L}_2^{-1})$ are nilpotent, the map
\[
	\pi_0\PMGL_S(X)[[x,y]] \simeq \pi_0\PMGL(\Pic_S\times\Pic_S) \to \pi_0\PMGL_S(X)
\]
induced by the map $X\to\Pic\times\Pic$ classifying $(\mcal{L}_1,\mcal{L}_2^{-1})$ factors through $\pi_0\PMGL_S(X)[[x,y]]/(x,y)^m$ for some $m>0$.
Since $c_1(\mcal{L})$ is the image of $f_\univ$, it is nilpotent.
Then the last claim is an immediate consequence of the construction.
\end{proof}

\begin{lemma}\label{lem:chern}
Suppose that $X\in\Sch_S$ is noetherian and admits an ample line bundle. %\footnote{In fact, $X$ can be an arbitrary qcqs derived scheme, cf.\ \cite[Lemma 4.2.3]{AI}.}
Let $\mcal{E}$ be a vector bundle of rank $r\ge 1$ on $X$.
Then the following hold in $\pi_0\PMGL_S(X)$:
\begin{enumerate}[label=\upshape{(\roman*)},leftmargin=*]
\item $e(\mcal{E})=c_r(\mcal{E})$.
\item $c_i(\mcal{E})$ is nilpotent for each $i\ge 1$.
\item If $\mcal{E}$ admits a filtration
\[
	0=\mcal{E}_0 \subset \mcal{E}_1 \subset \dotsb \subset \mcal{E}_r=\mcal{E}
\]
such that $\mcal{L}_i=\mcal{E}_i/\mcal{E}_{i-1}$ is a line bundle for $1\le i\le r$, then
\[
	c(\mcal{E}) = \prod_{i=1}^r(1+c_1(\mcal{L}_i)t).
\]
\item If we have a fiber sequence
\[
	\mcal{E}'\to \mcal{E}\to \mcal{E}''
\]
of vector bundles on $X$, then $c(\mcal{E})=c(\mcal{E}')\cdot c(\mcal{E}'')$.
\end{enumerate}
\end{lemma}
\begin{proof}
Firstly, (iv) follows from (iii) by taking the pullback of $\mcal{E}$ to the derived scheme representing full flags of $\mcal{E}$.
Similarly, (ii) follows from (iii) and Lemma \ref{lem:fgl}.
We need a splitting trick for the other assertions.
Suppose we are given a fiber sequence
\[
	\mcal{E}'\xrightarrow{f} \mcal{E}\to \mcal{E}''
\]
of vector bundles on $X$.
Let $t_0,t_1$ be homogeneous coordinates of $\mbb{P}^1$, and let $\tilde{\mcal{E}}$ be the cofiber of the map
\[
	(\mrm{id}\otimes t_0)\oplus(f\otimes t_1) \colon \mcal{E}' \to \mcal{E}'(1)\oplus \mcal{E}(1)
\]
of vector bundles on $\mbb{P}^1\times X$.
Then $i_0^*\tilde{\mcal{E}}=\mcal{E}'\oplus\mcal{E}''$ and $i_\infty^*\tilde{\mcal{E}}=\mcal{E}$, and thus we get $c(\mcal{E})=c(\mcal{E}'\oplus\mcal{E}'')$ and $e(\mcal{E})=e(\mcal{E}'\oplus\mcal{E}'')$.

(i) It is obvious for line bundles.
For the general case, we may assume that $\mcal{E}$ is a direct sum of line bundles $\bigoplus_{i=1}^r\mcal{L}_i$ by the splitting trick.
Since $\Sym(\bigoplus\mcal{L}_i)=\bigotimes_i\Sym(\mcal{L}_i)$, we have $e(\mcal{E})=\prod_ie(\mcal{L}_i)$, and thus the assertion follows from (iii).

(iii) We may assume that $\mcal{E}=\bigoplus_{i=1}^r\mcal{L}_i$.
Consider the universal quotient $\mcal{E}\to\mcal{O}(1)$ on $\mbb{P}(\mcal{E})$.
The induced map $\mcal{L}_i\to\mcal{O}(1)$ gives a global section $s_i$ of $\mcal{L}_i^{-1}(1)$, and let $D_i\subset\mbb{P}(\mcal{E})$ be the derived vanishing locus of $s_i$.
Then the intersection of all $D_i$ for $1\le i\le r$ is empty, from which it follows that $\prod_ic_1(\mcal{L}_i^{-1}(1))=0$.
By the formal group law (Lemma \ref{lem:fgl}),
\[
	\xi = c_1(\mcal{L}_i\otimes\mcal{L}_i^{-1}(1))
		= c_1(\mcal{L}_i)+c_1(\mcal{L}_i^{-1}(1)) + \sum_{p,q\ge 1}a_{pq}c_1(\mcal{L}_i)^pc_1(\mcal{L}_i^{-1}(1))^q
\]
for some $a_{pq}\in\pi_0\PMGL_S(S)$, and therefore $\prod_i(\xi-c_1(\mcal{L}_i))=0$.
This implies the desired equation.
\end{proof}

We can compute the cohomology of grassmannians in terms of Chern classes by using Lemma \ref{lem:chern}.

\begin{lemma}\label{lem:grass}
Let $E$ be a pbf-local sheaf with transfers on $\Sch_S$ and $n\ge 0$.
Let $\mcal{E}$ be a vector bundle of rank $r\ge 1$ on a derived scheme $X\in\Sch_S$ and $\mcal{Q}$ the universal quotient bundle of $\mcal{E}$ on $\Gr_n(\mcal{E})$.
Assume that $X=S\times_{S'}X'$ for some noetherian derived scheme $X'\in\Sch_{S'}$ admitting an ample line bundle and that $\mcal{E}=\mcal{E}'_X$ for some vector bundle $\mcal{E}'$ of rank $r$ on $X'$. %\footnote{In fact, this assumption is unnecessary, cf.\ \cite[Lemma 4.2.4]{AI}.}
Then the morphism
\[
	\sum_\alpha (c(\mcal{Q})^\alpha \cdot p^*) \colon \bigoplus_\alpha E(X) \to E(\Gr_n(\mcal{E}))
\]
is an equivalence, where $\alpha$ runs over all $n$-tuples of non-negative integers with $|\alpha|\le r-n$ and $c^\alpha:=\prod c_i^{\alpha_i}$.
\end{lemma}
\begin{proof}
We prove by induction on $n$.
The case $n=0$ is obvious, and let $n\ge 1$.
By Lemma \ref{lem:pbf}, the map
\[
	\sum_i \xi^i\colon \bigoplus_{i=0}^{n-1} E(\Gr_n(\mcal{E})) \to E(\mbb{P}(\mcal{Q}))
\]
is an equivalence.
Hence, it suffices to show that the composite
\[
	\bigoplus_i\bigoplus_\alpha E(X) \xrightarrow{c(\mcal{Q})^\alpha} \bigoplus_i E(\Gr_n(\mcal{E})) \xrightarrow{\xi^i} E(\mbb{P}(\mcal{Q}))
\]
is an equivalence.
Note that $\mbb{P}(\mcal{Q})$ is isomorphic to the grassmannian $\Gr_{n-1}(\mcal{E}')$ of the universal subbundle $\mcal{E}'$ of $\mcal{E}$ on $\mbb{P}(\mcal{E})$.
Therefore, by the induction hypothesis, the map
\[
	\sum_{j,\beta} c(\mcal{Q}')^\beta\xi^j \colon
	\bigoplus_\beta\bigoplus_{j=0}^{r-1} E(X) \to E(\mbb{P}(\mcal{Q}))
\]
is an equivalence, where $\mcal{Q}'$ is the universal subbundle of $\mcal{Q}$ on $\mbb{P}(\mcal{Q})$ and $\beta$ runs over all $(n-1)$-tuples of non-negative integers with $|\beta|\le r-n$.
By Lemma \ref{lem:chern} (iv), we have $c_k(\mcal{Q})=c_k(\mcal{Q}')+c_{k-1}(\mcal{Q}')\xi$.
Therefore, it remains to show that the family
\[
	\Bigl\{ \xi^i\textstyle\prod_k(c_k(\mcal{Q}')+c_{k-1}(\mcal{Q}')\xi)^{\alpha_k} \Bigr\}_{0\le i\le n-1,\alpha}
\]
forms the same linear system as $\{c(\mcal{Q}')^\beta\xi^j\}_{0\le j\le r-1,\beta}$, and this is elementary; or we can refer to the classical computation of cohomology of grassmannians (as in \cite[\S14.6]{Ful}, for example), since it gives a desired invertible matrix over $\mbb{Z}$.
\end{proof}

\begin{corollary}\label{cor:grass}
Let $E$ be a homotopy commutative algebra in $\SH^\tr_\pbf(\Sch_S)$.
Then there is a canonical ring isomorphism
\[
	\pi_*E(\Gr_{n,S}) \simeq \pi_*E(S)[[c_1,\dotsc,c_n]],
\]
where $c_i$ is the $i$-th Chern class of the universal quotient bundle on $\Gr_{n,S}$.
\end{corollary}
\begin{proof}
By Lemma \ref{lem:grass}, the map
\[
	\pi_*E(S)[c_1,\dotsc,c_n] \to \pi_*E(\Gr_n(\mcal{O}_S^N))
\]
sending $c_i$ to the $i$-th Chern class of the quotient bundle $\mcal{Q}$ is surjective, and if we let $I_N$ be its kernel then $\bigcap_NI_N=\varnothing$.
Since $\Gr_n(\mcal{O}^N_S)$ is a base change of $\Gr_n(\mcal{O}^N_\mbb{Z})$, Lemma \ref{lem:chern} (ii) implies that $c_i(\mcal{Q})\in\pi_*E(\Gr_n(\mcal{O}^N_S))$ is nilpotent for each $i\ge 1$.
Hence, $\pi_*E(\Gr_n(\mcal{O}_S^N))$ is complete for the $(c_1,\dotsc,c_n)$-adic topology and the above map factors through $\pi_*E(S)[[c_1,\dotsc,c_n]]$.
Taking limits with respect $N$, we obtain a map
\[
	\pi_*E(S)[[c_1,\dotsc,c_n]] \to \lim_N\pi_*E(\Gr_n(\mcal{O}_S^N)) \simeq \pi_*E(\Gr_{n,S}),
\]
which is injective since $\bigcap_NI_N=\varnothing$ and induces the identity modulo $(c_1,\dotsc,c_n)$.
Note that $\pi_*E(\Gr_{n,S})$ is separated for the $(c_1,\dotsc,c_n)$-adic topology since we have an injection
\[
	(c_1,\dotsc,c_n)^m \lim_N\pi_*E(\Gr_n(\mcal{O}_S^N)) \to \lim_N (c_1,\dotsc,c_n)^m\pi_*E(\Gr_n(\mcal{O}_S^N))
\]
and the right hand side is zero when $m\to\infty$.
Then the result follows from the following observation:
if $A$ is an $I$-adically complete ring and $B$ is an extension ring of $A$ such that $A/IA=B/IB$ and that $B$ is separated for the $I$-adic topology, then $A=B$.
Indeed, $B=A+IB=A+I(A+IB)=\dotsb$, and thus the separatedness implies the claim.
\end{proof}

Now we can complete the proof of the main theorem.

\begin{proof}[Proof of Theorem \ref{thm:vect}]
We may assume that the base scheme is $S=\Spec(\mbb{Z})$.
We first prove that the canonical map
\[
	L_\mot\gamma^*\Gr_n \to L_\mot\gamma^*\Vect_n
\]
admits a left inverse.
Let $\mcal{E}_\univ$ be the universal vector bundle of rank $n$ over $\Vect_n$.
Then the projective space $\mbb{P}(\mcal{E}_\univ)$ is defined as the stack classifying quotients $\mcal{E}_\univ\twoheadrightarrow\mcal{L}$ with $\mcal{L}$ being a line bundle.
Then the Euler class $\xi$ of the universal quotient bundle on $\mbb{P}(\mcal{E}_\univ)$ is defined as in Lemma \ref{lem:pic} and the map
\[
	\sum_{i=0}^{r-1}(\xi^i\cdot p^*)\colon \bigoplus_{i=0}^{r-1}E(\Vect_n) \to E(\mbb{P}(\mcal{E}_\univ))
\]
is an isomorphism for any pbf-local sheaf $E$ with transfers by Lemma \ref{lem:pbf}.
Hence, we can define the Chern classes $c_i$ of $\mcal{E}_\univ$ by the formula in Definition \ref{def:chern}.
Then we get a commutative diagram
\[
\xymatrix{
	\prod_\alpha (L_\mot\gamma^*\Gr_n)(\Spec(\mbb{Z})) \ar[d]_{\sum(c(\mcal{E}_\univ)^\alpha\cdot p^*)} \ar[rd]^\simeq & \\ 
	(L_\mot\gamma^*\Gr_n)(\Vect_n) \ar[r] &
	(L_\mot\gamma^*\Gr_n)(\Gr_n)
}
\]
and the diagonal arrow is an equivalence by Lemma \ref{lem:grass}.
Hence, the canonical map $\gamma^*\Gr_n\to L_\mot\gamma^*\Gr_n$ lifts to a map $\gamma^*\Vect_n\to L_\mot\gamma^*\Gr_n$, which gives a desired left inverse.

The rest of the proof is identical to that of Corollary \ref{cor:pic}.
By Corollary \ref{cor:key}, we have a commutative diagram
\[
\xymatrix{
	& L_\Zar(\gamma_*L_\mot\gamma^*\Gr_n)_{\le k} \ar[d] \\
	\Vect_n \ar[r] \ar@/^1pc/[ru]^-{\phi_k} & L_\Zar(\gamma_*L_\mot\gamma^*\Vect_n)_{\le k}.
}
\]
The map $\phi_k$ is characterized as a unique map which makes the diagram commutative since the right vertical map has a left inverse.
In particular, these maps are assembled into a map
\[
	\phi \colon \Vect_n \to \lim_k L_\Zar(\gamma_*L_\mot\gamma^*\Gr_n)_{\le k} \simeq \gamma_*L_\mot\gamma^*\Gr_n,
\]
which induces a right inverse of the canonical map $L_\mot\gamma^*\Gr_n\to L_\mot\gamma^*\Vect_n$.\footnote{As in the proof of Corollary \ref{cor:pic}, it is not clear if the induced map is indeed a right inverse. This gap has been fixed in \cite{AI}, and the result here remains true.}
Since we have seen that the canonical map has a left inverse, we conclude that it is an equivalence.
\end{proof}

\begin{corollary}\label{cor:vect}
Let $E$ be a homotopy commutative algebra in $\SH^\tr_\pbf(\Sch_S)$.
Then there is a canonical ring isomorphism
\[
	\pi_*E(\Vect_{n,S}) \simeq \pi_*E(S)[[c_1,\dotsc,c_n]],
\]
where $c_i$ is the $i$-th Chern class of the universal vector bundle of rank $n$.
\end{corollary}
\begin{proof}
This follows from Theorem \ref{thm:vect} and Corollary \ref{cor:grass}.
\end{proof}

\section{Representability of localizing invariants}\label{loc}

Examples of pbf-local sheaves of spectra with transfers are supplied by localizing invariants.
Let $\Cat^\ex_\infty$ be the $\infty$-category of small stable $\infty$-categories and exact functors.
A \textit{localizing invariant} is a functor $\Cat^\ex_\infty\to\Sp$ which carries exact sequences in $\Cat^\ex_\infty$ to fiber sequences of spectra, cf.\ \cite[Definition 1.2]{LT}.

Let $\Sch$ be the $\infty$-category of all qcqs derived schemes and suppose that it is marked with respect to finite quasi-smooth morphisms in the sense of \cite[Definition 2.6.1]{Mac}.
Then the functor $\Perf\colon\Sch^\op\to\Cat_\infty^\ex$ of perfect complexes is bivariant in the sense of \cite[\S3.2]{Mac} by \cite[Corollary 3.2.3.3, Theorem 6.1.3.2]{SAG}.
Hence, by \cite[Theorem 3.8.1]{Mac}, the functor $\Perf$ extends to a functor
\[
	\Perf^\dagger \colon \Corr^\fqsm(\Sch)^\op \to \Cat_\infty^\ex.
\]

\begin{lemma}\label{lem:inv}
If $E$ is a localizing invariant and $S$ a qcqs derived scheme, then the composite
\[
	\Corr^\fqsm(\Sch_S)^\op \to \Corr^\fqsm(\Sch)^\op \xrightarrow{\Perf^\dagger} \Cat_\infty^\ex \xrightarrow{E} \Sp
\]
is a pbf-local sheaf of spectra with transfers on $\Sch_S$.
\end{lemma}
\begin{proof}
It suffices to show that the composite $E\circ\Perf^\dagger$ satisfies Zariski descent and projective bundle formula.
The descent essentially follows from the work \cite{TT}, see also \cite[Lemma A.1]{LT}, and the projective bundle formula is satisfied by \cite[Theorem B]{Kha}.
\end{proof}

We regard a localizing invariant as a pbf-local sheaf with transfers by Lemma \ref{lem:inv}.
The next goal is to show that a multiplicative localizing invariant yields an $\mbb{E}_\infty$-algebra in $\SH^\tr_\pbf(\Sch_S)$.
A \textit{multiplicative localizing invariant} is a lax symmetric monoidal functor $(\Cat^\ex_\infty)^\otimes \to \Sp^\otimes$ whose underlying functor $\Cat^\ex_\infty\to\Sp$ is a localizing invariant.

\begin{lemma}\label{lem:mult}
Let $\mcal{C}^\otimes$ be a cocartesian symmetric monoidal $\infty$-category whose underlying $\infty$-category $\mcal{C}$ is equipped with a marking with collar change \cite[3.1.3]{Mac}\footnote{I.e., $\mcal{C}$ is equipped with a distinguished class of morphisms which are stable under pushouts along arbitrary morphisms.}.
Let $F\colon\mcal{C}^\otimes\to\Cat_\infty$ be a lax cartesian structure \cite[Definition 2.4.1.1]{HA}.
Assume that:
\begin{enumerate}[label=\upshape{(\roman*)},leftmargin=*]
\item The underlying functor $F\colon\mcal{C}\to\Cat_\infty$ is right bivariant with collar change \cite[3.2.5]{Mac}.
\item $F$ satisfies projection formula, i.e., for any marked morphism $f\colon X\to Y$ in $\mcal{C}$, the canonical map
\[
	f_*(-)\otimes(-) \to f_*((-)\otimes f^*(-))
\]
is an equivalence, where $f^*:=F(f)$ and $f_*$ is a right adjoint of $f^*$.
\end{enumerate}
Then $F$ extends to a lax cartesian structure $F^\dagger\colon\co\Corr(\mcal{C})^\otimes\to\Cat_\infty$, where $\co\Corr(\mcal{C})$ is the $\infty$-category of cocorrespondences in $\mcal{C}$, i.e., $\co\Corr(\mcal{C}):=\Corr(\mcal{C}^\op)^\op$.
\end{lemma}
We remark that the functor $F\colon\mcal{C}\to\Cat_\infty$ factors through the $\infty$-category $\Mon_{\mbb{E}_\infty}(\Cat_\infty)$ of symmetric monoidal $\infty$-categories, but $F^\dagger\colon\co\Corr(\mcal{C})\to\Cat_\infty$ does not factor through $\Mon_{\mbb{E}_\infty}(\Cat_\infty)$ in general.
\begin{proof}
We equip $\mcal{C}^\otimes$ with the vertical marking, i.e., a morphism $f$ in $\mcal{C}^\otimes$ is marked if and only if it lies over the identity $\langle n\rangle\to\langle n\rangle$ for some $n\ge 0$ and $f$ is a product of marked morphisms in $\mcal{C}$.
It suffices to show that the functor $F\colon\mcal{C}^\otimes\to\Cat_\infty$ is right bivariant with collar change; then we can apply Macpherson's extension theorem \cite[Theorem 3.8.1]{Mac} and the extension is lax cartesian since $\mcal{C}^\otimes\to\co\Corr(\mcal{C})^\otimes$ is essentially surjective.
We have to show that, for a cocartesian square
\[
\xymatrix{
	(Y_1,\dotsc,Y_m) \ar[d]^g \ar[r]^f & (X_1,\dotsc,X_m) \ar[d]^{g'} \\
	(Y_1',\dotsc,Y_n') \ar[r]^{f'} & (X_1',\dotsc,X_n')
}
\]
in $\mcal{C}^\otimes$ with $f$ being a product of marked morphisms, the base change map $g^*f_*\to f'_*g^{'*}$ is an equivalence.
This is clear if $g$ is an inert morphism and thus we may assume that $g$ is active.
Then, by induction, the problem is reduced to the right adjointability of cocartesian squares of the form
\[
\xymatrix{
	(Y,Y) \ar[d]^\delta \ar[r]^{(f_1,f_2)} & (X_1,X_2) \ar[d]^{\delta'} \\
	Y \ar[r]^f & X,
}
\]
where $\delta$ is the codiagonal map.
The square is cocartesian if and only if $X\simeq X_1\sqcup_YX_2$, and in this case the base change map $\delta^*(f_1,f_2)_*\to f_*\delta^{'*}$ is identified with $f_{1*}\otimes f_{2*}\to f_*(f_1^{'*}\otimes f_2^{'*})$, where $f_i'$ is the canonical map $X_i\to X$.
Then it is an equivalence by the projection formula.
\end{proof}

\begin{lemma}\label{lem:ex}
Let $\mcal{C}^\otimes$ be a symmetric monoidal $\infty$-category and $F\colon\mcal{C}^\otimes\to\Cat_\infty$ a lax cartesian structure.
Assume that:
\begin{enumerate}[label=\upshape{(\roman*)},leftmargin=*]
\item The underlying functor $F\colon\mcal{C}\to\Cat_\infty$ factors through the subcategory $\Cat_\infty^\ex$.
\item For each morphism $(X_1,\dotsc,X_n)\to X$ in $\mcal{C}^\otimes$ lying over the active morphism $\langle n\rangle\to\langle 1\rangle$, the induced functor
\[
	F(X_1)\times\dotsb\times F(X_n) \to F(X)
\]
preserves finite colimits in each variable.
\end{enumerate}
Then $F$ uniquely lifts to a lax symmetric monoidal functor $\mcal{C}^\otimes\to(\Cat_\infty^\ex)^\otimes$.
\end{lemma}
\begin{proof}
By \cite[Proposition 2.4.1.7]{HA}, the $\infty$-category $\Fun^\lax(\mcal{C}^\otimes,\Cat_\infty)$ of lax cartesian structures is equivalent to the $\infty$-category $\Fun^\lax(\mcal{C}^\otimes,\Cat_\infty^\times)$ of lax symmetric monoidal functors.
Under this equivalence, a lax symmetric monoidal functor $F\colon\mcal{C}^\otimes\to\Cat_\infty^\times$ factors through the subcategory $\Cat_\infty(\mscr{K})^\otimes$ constructed in \cite[Notation 4.8.1.2]{HA}, where we take $\mscr{K}$ to be the set of all finite simplicial sets, if and only if the $\infty$-category $F(X)$ admits finite colimits for each $X\in\mcal{C}$ and $F$ satisfies the condition (ii).
Furthermore, since $\Cat^\ex_\infty$ is a reflective subcategory of $\Cat_\infty(\mscr{K})$ and the left adjoint is symmetric monoidal, we conclude that $F$ uniquely factors thought a lax symmetric monoidal functor $\mcal{C}^\otimes\to(\Cat_\infty^\ex)^\otimes$.
\end{proof}

\begin{corollary}\label{cor:mult}
Let $E$ be a multiplicative localizing invariant and $S$ a qcqs derived scheme.
Then the associated pbf-local sheaf $E$ of spectra with transfers is promoted to an $\mbb{E}_\infty$-algebra in $\SH^\tr_\pbf(\Sch_S)$.
\end{corollary}
\begin{proof}
We apply Lemma \ref{lem:mult} to $\mcal{C}=\Sch^\op$ and $F=\Perf$.
Then $\Perf$ satisfies projection formula by \cite[Remark 3.4.2.6]{SAG}, and thus we get a lax cartesian structure
\[
	\Perf^\dagger\colon(\Corr^\fqsm(\Sch)^\op)^\otimes\to\Cat_\infty.
\]
This extension satisfies the condition of Lemma \ref{lem:ex}, and thus $\Perf^\dagger$ lifts to a lax symmetric monoidal functor to $(\Cat_\infty^\ex)^\otimes$.
Then we get a lax symmetric monoidal functor
\[
	E^\otimes\colon\: (\Corr^\fqsm(\Sch_S)^\op)^\otimes
	\to (\Corr^\fqsm(\Sch)^\op)^\otimes
	\xrightarrow{\Perf^\dagger} (\Cat_\infty^\ex)^\otimes
	\xrightarrow{E} \Sp^\otimes.
\]
We can regard $E^\otimes$ as an $\mbb{E}_\infty$-algebra in the symmetric monoidal $\infty$-category $\Fun(\Corr^\fqsm(\Sch_S)^\op,\Sp)^\otimes$ of presheaves of spectra and the Day convolution products.
Since the localization functor $L_\mot$ is symmetric monoidal and $E^\otimes$ underlies the pbf-local sheaf $E$ of spectra with transfers, we conclude that $E=L_\mot E^\otimes$ is an $\mbb{E}_\infty$-algebra in $\SH^\tr_\pbf(\Sch_S)$.
\end{proof}

\end{document}